\documentclass[a4paper,12pt]{article}

\usepackage{amssymb}
\usepackage{latexsym}
\usepackage{amsfonts}
\usepackage{amsmath}
\usepackage{amsthm}
\usepackage{hyperref}
\usepackage{enumerate}
\usepackage{tensor}
\usepackage{mathtools}
\usepackage{stmaryrd}
\usepackage{mdwlist}
\usepackage{pgf,tikz,pgfplots}
\pgfplotsset{width=7.5cm,compat=1.9}

\usetikzlibrary{arrows,patterns,positioning}
\usetikzlibrary{calc}

\usepackage[margin=1in]{geometry}

\usepackage[compact]{titlesec}

\usepackage[scale=0.9]{tgheros}
\usepackage[T1]{fontenc}

\DeclareMathOperator{\Aut}{Aut}

\DeclareMathOperator{\SL}{SL}

\DeclareMathOperator{\pg}{PG}

\DeclareMathOperator{\Sp}{Sp}

\DeclareMathOperator{\GL}{GL}

\DeclareMathOperator{\GaL}{\Gamma L}

\DeclareMathOperator{\Sym}{Sym}

\DeclareMathOperator{\s}{S}

\renewcommand{\b}{\mathbf}
\renewcommand{\leq}{\leqslant}
\renewcommand{\geq}{\geqslant}
\renewcommand{\unlhd}{\trianglelefteqslant}
\newcommand{\ra}{\rightarrow}

\newcommand{\F}{\mathbb F}

\renewcommand{\L}{\mathcal L}

\renewcommand{\P}{\mathcal P}
\renewcommand{\S}{\mathcal S}

\newcommand{\V}{\mathcal V}
\newcommand{\W}{\mathsf W}
\newcommand{\E}{\mathcal E}

\theoremstyle{plain}
\newtheorem{lemma}{Lemma}
\newtheorem{theorem}[lemma]{Theorem}
\newtheorem{proposition}[lemma]{Proposition}
\newtheorem{corollary}[lemma]{Corollary}

\theoremstyle{definition}
\newtheorem{definition}[lemma]{Definition}
\newtheorem{example}[lemma]{Example}

\numberwithin{equation}{section}
\numberwithin{lemma}{section}

\begin{document}

\title{Symmetries of regular $q$-graphs
}
\author{
Daniel R. Hawtin\footnotemark[1]~ and 
Padraig \'{O} Cath\'{a}in\footnotemark[2]
}
\date{
\today
}

\maketitle
\begin{abstract}
 Given a finite vector space $V=\F_q^n$, the $q$-analogue of a graph, called a \emph{$q$-graph}, is a pair $\Gamma=(\V,\E)$, where $\V$ is the set of $1$-dimensional subspaces of $V$ and $\E$ is a subset of the $2$-dimensional subspaces of $V$. Elements of $\V$ and $\E$ are called \emph{vertices} and \emph{edges}, respectively. If the edges through a vertex $X$ consist of all $2$-spaces of a $(k+1)$-dimensional space which contain $X$, regardless of the choice of vertex, then $\Gamma$ is \emph{$k$-regular}. Moreover, $\Gamma$ is \emph{flag-transitive} if there is a subgroup of $\GaL_n(q)$ preserving $\E$ and acting transitively on the set of all incident vertex-edge pairs; and \emph{symmetric} if there is a subgroup of $\GaL_n(q)$ preserving $\E$ and acting transitively on the set of all ordered pairs of adjacent vertices. 
 
 This paper classifies all $k$-regular $q$-graphs that are either flag-transitive or symmetric. The $q$-graphs in the classification are constructed from familiar objects in finite geometry, including spreads, symplectic polar spaces, and generalised hexagons. The classification depends essentially on the classification of transitive linear groups, and thus ultimately on the classification of finite simple groups. 
\end{abstract}

\renewcommand{\thefootnote}{\fnsymbol{footnote}}
\footnotetext[1]{Faculty of Mathematics, University of Rijeka, 51000 Rijeka, Croatia. Email: \href{mailto:dhawtin@math.uniri.hr}{dhawtin@math.uniri.hr}}
\footnotetext[2]{Údarás na Gaeltachta, Na Forbacha, Co. na Gaillimhe, formerly Fiontar \& Scoil na Gaeilge, Dublin City University, Dublin 9, Ireland. Email: \href{mailto:p.ocathain@gmail.com}{p.ocathain@gmail.com}}
\renewcommand{\thefootnote}{\roman{footnote}}

\section{Introduction}

Let $V = \F_{q}^{n}$ be a vector space over a finite field. Denote by $\mathcal{V}$ the set of one-dimensional subspaces of $V$, we reserve $\mathcal{E}$ for a subset of the 2-dimensional subspaces of $V$. Subspaces of $V$ are typically denoted by capital letters, thus we may refer to $U, W \in \mathcal{V}$ and to $E, F \in \mathcal{E}$, while lower case Latin letters will typically be vectors $v \in V$. 

A \emph{$q$-graph} $\Gamma$ is a pair $(\V,\E)$ where $\V$ and $\E$ are as above. The elements of $\V$ and $\E$ are called \emph{vertices} and \emph{edges}, respectively. 

If a vertex $X$ is contained in the edge $U$, we say that $X$ and $U$ are \emph{incident}, or that $(X, U)$ is a \emph{flag}. Distinct vertices $X,Y\in\V$ are \emph{adjacent} if their span $\langle X,Y\rangle$ is an edge. If the edges through a vertex $X$ consist of all $2$-spaces of a $(k+1)$-dimensional space which contain $X$, then $X$ is said to have \emph{$q$-degree} $k$. A $q$-graph $\Gamma$ is \emph{$k$-regular} if all vertices have $q$-degree $k$. 

The \emph{automorphism group} $\Aut(\Gamma)$ of $\Gamma$ is the setwise stabiliser of $\E$ in $\GaL_n(q)$. Let $G\leq \Aut(\Gamma)$. We say that $\Gamma$ is \emph{$G$-vertex-transitive} if $G$ acts transitively on $\V$, and \emph{$G$-edge-transitive} if $G$ acts transitively on $\E$. 

\begin{definition}\label{def:sym}
 Let $\Gamma=(\V,\E)$ be a $q$-graph on $V=\F_q^n$, and let $G\leq \Aut(\Gamma)$. 
 \begin{enumerate}[(1)]
  \item $\Gamma$ is \emph{$G$-flag-transitive} if $G$ acts transitively on the set of all pairs $(X,U)\in\V\times \E$ such that $X\leq U$.
  \item $\Gamma$ is \emph{$G$-symmetric} if $G$ acts transitively on the set of all pairs $(X,Y)\in\V\times \V$ such that $\langle X,Y\rangle\in\E$.
 \end{enumerate}
 If $\Gamma$ is $\Aut(\Gamma)$-flag-transitive or $\Aut(\Gamma)$-symmetric then we simply say that $\Gamma$ is \emph{flag-transitive} or \emph{symmetric}, respectively.
\end{definition}
For any $q$-graph, $G$-symmetric implies $G$-flag-transitive, and $G$-flag-transitive implies $G$-edge-transitive. If in addition $\Gamma$ is $k$-regular, with $k\geq 1$, and $G$-flag-transitive or $G$-symmetric, then every vertex is contained in some edge, so $\Gamma$ is also $G$-vertex-transitive. In this paper, we study $k$-regular $q$-graphs that are $G$-flag-transitive or $G$-symmetric or both. We state our main results below. 

\begin{theorem}\label{thm:main} 
 Suppose that $\Gamma=(\V,\E)$ is a $k$-regular and $G$-flag-transitive $q$-graph on $V=\F_q^n$ that is neither complete nor empty. 
 Then one of the following occurs. 
 \begin{enumerate}[$(1)$]
  \item $n$ is even, $k=1$, 
  and $\E$ induces a partition of $\V$, as in Example~\ref{ex:lineSpeads}. 
  \item $n\equiv 0\pmod{3}$, $k=2$, $q\in\{2,8\}$, and $\Gamma$ is a union of complete $q$-graphs, as in Example~\ref{ex:spreadsByProjPlanes}. 
  \item $n$ is even, $k=n-2$, $\Sp_n(q) \unlhd G$ and $\E$ is the set of lines of a symplectic polar space, as in Example~\ref{ex:symplecticPolarSpaces}. 
  \item $n=6$, $k=2$, $q$ is even, $G_{2}(q)' \unlhd G$, and $\E$ is the set of lines of a symplectic generalised hexagon, as in Example~\ref{ex:symplecticHexagons}. 
 \end{enumerate} 
\end{theorem} 

\begin{corollary}\label{cor:main} 
 Suppose that $\Gamma=(\V,\E)$ is a $k$-regular and $G$-symmetric $q$-graph on $V=\F_q^n$ that is neither complete nor empty. 
 Then one of the following occurs. 
 \begin{enumerate}[$(1)$]
  \item $q=2$, $n$ is even, $k=1$, $G$ acts transitively on $\V$ and $\E$ is a Desarguesian spread of $V$.
  \item $n$ is even, $k=n-2$, $\Sp_n(q) \unlhd G$ and $\E$ is the set of lines of a symplectic polar space, as in Example~\ref{ex:symplecticPolarSpaces}. 
  \item $n=6$, $k=2$, $q$ is even, $G_{2}(q)' \unlhd G$, and $\E$ is the set of lines of a symplectic generalised hexagon, as in Example~\ref{ex:symplecticHexagons}. 
 \end{enumerate} 
\end{corollary} 

More information about the possibilities for $\E$ in Theorem~\ref{thm:main}(1) is given in Example~\ref{ex:lineSpeads}; in particular, existing results in the literature show that $\E$ necessarily forms a Desarguesian spread, except in the cases $(n,q)=(4,3),\,(6,3)$. The remaining $q$-graphs arising under Theorem~\ref{thm:main} are given in Examples~\ref{ex:spreadsByProjPlanes}--\ref{ex:symplecticHexagons}. The $q$-graphs in Theorem~\ref{thm:main}(3) have been studied as strongly-regular $q$-graphs in \cite{braun2023q}, and the case $q=2$ in Theorem~\ref{thm:main}(4) has been studied as a Deza $q$-graph in \cite{crnkovic2025qdeza}. To the best of our knowledge, the examples appearing in Theorem~\ref{thm:main}(4) with $q>2$ have not previously been studied as $q$-graphs. 

Examples of $q$-graphs that satisfy all but one of the hypotheses of Theorem~\ref{thm:main} are given in Section~\ref{sec:furtherExamples}. In particular, the $q$-graphs constructed in Example~\ref{ex:spreadComplements} are flag-transitive, but not regular; those constructed in Example~\ref{ex:subfields} are regular and vertex-transitive, but (in general) not flag-transitive; and Example~\ref{ex:EdgeTransNotFlagTrans} provides an example of a $q$-graph that is regular and edge-transitive, but not flag-transitive.

The definition of a $q$-graph, with notions such as the degree of a vertex and the definition of the automorphism group, was introduced in 2023 by Braun, Crnkovi\'c, De Boeck, Mikuli\'c Crnkovi\'c and \v{S}vob~\cite{braun2023q}. Their main result was the classification of $q$-analogues of strongly regular graphs. Since then, regular $q$-graphs have been studied in several other papers. Crnkovi\'c, De Boeck, Pavese and \v{S}vob constructed $q$-analogues of Deza graphs and characterised the smallest non-strongly-regular examples \cite{crnkovic2025qdeza}; Crnkovi\'c, Mikuli\'c Crnkovi\'c, \v{S}vob and Zubovi\'c \v{Z}utolija proposed a method of constructing vertex-transitive regular $q$-graphs \cite{crnkovic2025trans}. Buratti, Naki\'c and Wasserman have studied graph decompositions in projective geometries, a topic related to the $q$-analogues of group divisible designs \cite{buratti2,buratti}. Our contribution to the study of $q$-graphs is the classification of $q$-graphs under the symmetry conditions of Definition~\ref{def:sym}.

\subsection{Groups acting on classical graphs}

Before moving on, we will briefly compare and connect our results with those of groups acting on graphs in the classical setting. Recall that a (classical) graph is called \emph{$G$-symmetric} (or simply \emph{symmetric}) if it admits a group $G$ of automorphisms that acts transitively on the set of all ordered pairs of adjacent vertices. In the classical setting, this condition is also called \emph{$G$-arc-transitive}, or simply \emph{arc-transitive}. The class of arc-transitive graphs is too vast to be classified. Major results in the area typically impose additional conditions, e.g. on the action of the arc-stabiliser of $G$ \cite{potovcnik2014order}, or on the valency \cite{conder2009refined}. A graph is \emph{$2$-arc-transitive} if it admits a group of automorphisms acting transitively on the triples $(u,v,w)$ of vertices such that $u$ is adjacent to $v$ and $v$ is adjacent to $w$, but $u\neq w$. While much progress has been made in understanding this class of graphs in terms of `normal quotients' and `basic' $2$-arc-transitive graphs, classification is thought to be impossible even in this case \cite{ivanov1993finite, praeger1993nan}. 

Considering the vastness of the family of regular, symmetric graphs, one might expect that the family of symmetric (or flag-transitive) $q$-graphs might also elude classification. From this perspective, Theorem~\ref{thm:main} and Corollary~\ref{cor:main} might seem surprising. However, while transitive groups on a finite set of size $n$ contain the regular actions of all finite groups of order $n$, and hence are not reasonably classifiable, the class of transitive linear groups is small, see Theorem \ref{thm:linclass}. Also, the automorphism groups of $q$-analogues of combinatorial objects tend to be much `smaller' than that of their classical counterparts, and hence conditions that may seem mild in the classical literature can be  restrictive for $q$-analogues (see, for instance, \cite{hawtin2022non}).

\subsection{Connection to classical graphs and geometries}

Given a $q$-graph $\Gamma=(\V,\E)$ on $V=\F_q^n$, define its classical counterpart $\Gamma_{C}=(\V_{C},\E_{C})$, where $\V_{C}=\V$ and $\E_{C}=\{\{u,v\}\in\V\times\V\mid \langle u,v\rangle\in\E\}$.  Observing that $\E_C$ is precisely the set of all pairs of adjacent vertices, we deduce the following result, which shows that the $q$-graphs in Corollary~\ref{cor:main} have classical counterparts that are symmetric graphs. 

\begin{corollary}\label{cor:classicalSymmetric}
 If $\Gamma$ is a $G$-symmetric $q$-graph, then its classical counterpart is a $G$-symmetric graph.
\end{corollary}

While the above result follows directly from the definitions, the result below does not appear to have a straightforward proof, and relies on Theorem~\ref{thm:main}. The proof is given in Section~\ref{sec:mainproof}.

\begin{proposition}\label{prop:CCsymmetric}
 If $\Gamma$ is a regular, flag-transitive $q$-graph, then its classical counterpart is a symmetric graph.
\end{proposition}

Note that while the above proposition shows that every flag-transitive $q$-graph $\Gamma$ has a symmetric classical counterpart $\Gamma_C$, the following example shows that the proposition is false if we insist upon only using automorphisms of $\Gamma_C$ that arise from elements of $\Aut(\Gamma)$. That is, there exist $G$-flag-transitive $q$-graphs $\Gamma$ such that $\Gamma_C$ is not $G$-symmetric.

\begin{example}
 If $\Gamma$ is as in Example~\ref{ex:spreadsByProjPlanes} with $q=2$ and $G=\GaL_1(2^n)$, then $\Gamma_C$ is not $G$-symmetric, since the faithful action of the stabiliser of a vertex $X$ on its neighbourhood has order $3$, while $N(X)$ contains $6$ vertices, not including $X$ itself. However, since $\Gamma_C$ is a union of cliques, it is $\Aut(\Gamma_C)$-symmetric.
\end{example}

Recall that a \emph{partial linear space} is an incidence structure in which any pair of points is contained in $0$ or $1$ lines. Partial linear spaces with a high degree of symmetry have been studied by Devillers \cite{devillers2000d}, and those admitting a rank $3$ group of automorphisms have been studied by Bamberg \cite{bamberg2021partial}. A partial linear space admitting a rank $3$ automorphism group is both flag-transitive and antiflag-transitive. Since distinct vertices in a $q$-graph $\Gamma$ span a unique $2$-space, which either is or is not an edge of $\Gamma$, the graphs of Theorem~\ref{thm:main} provide examples of flag-transitive partial linear spaces.

\begin{proposition}
 Let $\Gamma=(\V,\E)$ be a $G$-flag-transitive $q$-graph. Then $\P=(\V,\E,{\rm I})$ is a $G$-flag-transitive partial linear space, where the incidence relation ${\rm I}$ is given by symmetrised containment.
\end{proposition}

\subsection{Outline of the paper}

Section~\ref{sec:pre} contains preliminaries, most importantly a statement of the classification of transitive linear groups. In Section~\ref{sec:examples}, we construct the $q$-graphs arising under Theorem~\ref{thm:main} and Corollary~\ref{cor:main}, as well as several $q$-graphs satisfying weaker hypotheses than those of Theorem~\ref{thm:main}. Section~\ref{sec:semi} contains the analysis of one-dimensional semi-linear actions, upon which results in later sections rely, and Section~\ref{sec:insol} treats the transitive linear groups of Lie type. The proofs of the main results are then given in Section~\ref{sec:mainproof}.

\section{Preliminaries}\label{sec:pre}

Let $G$ be a group acting on a set $\Omega$ and let $\Delta\subseteq\Omega$. We denote the action of $g \in G$ on $\alpha \in \Omega$ by $\alpha^g$. For $\alpha\in\Omega$ we write $G_\alpha$ for the stabiliser of $\alpha$, and we write $G_\Delta$ for the setwise stabiliser of $\Delta$, inside $G$. If $H\leq G_\Delta$, then we write $H^\Delta$ for the \emph{faithful} action of $H$ on $\Delta$, so that $H^\Delta\leq\Sym(\Delta)$ and $H^\Delta\cong H/K$, where $K$ is the kernel of the action of $H$ on $\Delta$. For more on permutation groups see, for instance, Dixon \cite{dixon1996permutation}. For more information about the construction of particular groups used throughout the paper, see Wilson \cite{wilson2009finite}. 

Let $\Gamma = (\mathcal{V}, \mathcal{E})$ be a $q$-graph. For any vertex $X \in \mathcal{V}$, the (closed) \emph{neighbourhood} $N(X)$ is defined to be  $\bigcup\{Y\in \E\mid X\leq Y\}$. Note that the elements of $N(X)$ are vectors, rather than vertices as in the original definition, cf. \cite[Definition~3.3]{braun2023q}.  Because we frequently make use of the fact that the neighbourhoods in regular $q$-graphs are vector spaces, we take their elements to be vectors and consider the vertices and edges contained in a neighbourhood to be contained as subspaces.

When $q = p^{t}$, the Galois group of $\mathbb{F}_{q}$ over $\mathbb{F}_{p}$ is cyclic of order $t$, generated by the map $\sigma: x \mapsto x^{p}$. The group $\GaL_{n}(\mathbb{F}_q)$ is defined to be the semi-direct product with normal subgroup $\GL_{n}(\mathbb{F}_q)$ and complement $\langle \sigma\rangle$, acting entrywise on matrices. The group $\GaL_{n}(\mathbb{F}_{q})$ has an induced action on the $k$-dimensional subspaces of $\mathbb{F}_{q}^{n}$, and the automorphism group of $\Gamma$ is defined to be the setwise stabiliser of $\mathcal{E}$. 

A \emph{$k$-spread} $\S$ of $V=\F_q^n$ is a set of $k$-dimensional subspaces of $V$ such that every non-zero vector is contained in precisely one element of $\S$. It is well known that a $k$-spread of $V$ exists if and only if $k$ divides $n$ \cite{segre}. 

\begin{definition}
 Let $k$ be a divisor of $n$, let $F=\F_q$, and let $K=\F_{q^k}$, so that $K/F$ is a field extension of degree $k$. Let $U = K^{n/k}$ and $V=F^n$. Then, as $F$-vector spaces, $U$ and $V$ are isomorphic. Fixing such an isomorphism, the set of images in $V$ of all the one-dimensional $K$-subspaces of $U$ forms a $k$-spread in $V$. Such spreads are called \emph{regular} or \emph{Desarguesian}.
\end{definition}

A group of automorphisms of a regular $k$-spread $\S$ of $V$, must preserve the $\F_{q^k}$-vector spaces of $U$. It follows that the stabiliser of $\S$ in $\GaL_n(q)$ is $\GaL_{n/k}(q^k)$. 
Conversely, given a copy of $\GL_{n/k}(q^k)$ embedded in $\GL_{n}(q)$ acting on $V'=\F_q^n$, the set of orbits of the centre of $\GL_{n/k}(q^k)$ on $V'\setminus\{0\}$ correspond to the elements of a regular $k$-spread.

A \emph{(linear) representation} of a group $G$ on a vector space $V$ over a field $F$ is a homomorphism $G\ra GL(V)$, and is equivalent to saying that $G$ acts on $V$ and preserves the $F$-vector space structure of $V$. Equivalently, a representation of $G$ on $V$ induces the structure of an $FG$-module on $V$. An $FG$-module $V$ is irreducible if and only if the only $G$-invariant subspaces are $0$ and $V$. The following lemma is crucial in the proofs in later sections; given an action of the group $G$ it provides strong restrictions on the possibilities for $N(X)/X$ as a $G_{X}$-module.

\begin{lemma}\label{lem:transOnQuotientSpace}
Let $\Gamma$ be a $k$-regular, $G$-flag-transitive $q$-graph on $V=\F_{q}^n$, and let $X\in\V$. Then the neighbourhood $N(X)$ is an $\F_q G_X$-module, where $G_X$ is the stabiliser of $X$ in $G$. Moreover, $G_X$ induces a transitive action on the set of $1$-dimensional subspaces of the quotient $N(X)/X$, and hence $N(X)/X$ is an irreducible $\F_qG_X$-module. 
\end{lemma}

\begin{proof}
 Since $\Gamma$ is $k$-regular, we have that $N(X)$ is a subspace of $V$. Moreover, $G_X$ leaves $N(X)$ invariant, and hence $N(X)$ is an $\F_q G_X$-module. Since $\Gamma$ is flag-transitive, $G_X$ acts transitively on the set of $\frac{q^{k}-1}{q-1}$ distinct $2$-dimensional subspaces of $N(X)$ containing $X$, and hence induces a transitive action on the $1$-spaces of $N(X)/X$. This implies that there is no proper subspace of $N(X)/X$ preserved by $G_X$, and hence that $N(X)/X$ is irreducible as an $\F_q G_X$-module. 
\end{proof}

The next lemma considers the complete $q$-graph, and when it is flag-transitive under a one-dimensional semi-linear group. The result relies on the classification of groups acting transitively on the set of incident point-line pairs of a projective plane, see \cite{higman1962flag} or \cite[Theorem~1]{higman1961geometric}.

\begin{lemma}\label{lem:completeSemiLinear}
 Let $n\geq 3$ and let $\Gamma=(\V,\E)$ be the complete $q$-graph on the $n$-dimensional vector-space $V=\F_{q^n}$ over $\F_q$ and let $G\leq\GaL_1(q^n)$. Then $\Gamma$ is $G$-flag-transitive if and only if $n=3$, $q\in\{2,8\}$ and $G=\GaL_1(q^3)$.
\end{lemma}

\begin{proof}
 The set $\V$ forms the point-set of a projective space $\pg_{n-1}(q)$, and $\E$ forms the set of lines of $\pg_{n-1}(q)$. Thus, the set of flags of $\Gamma$ is precisely the set of point-line flags of $\pg_{n-1}(q)$. Let $q=p^t$, where $p$ is prime. If $n\geq 4$, then the number $(q^{n-1}+\cdots+ q+1)(q^{n-2}+\cdots +q+1)$ of flags of $\Gamma$ is strictly greater than $|\GaL_1(q^n)|=nt(p^{nt}-1)$, and hence $G$ is not flag-transitive on $\Gamma$. Thus we may assume that $n=3$. The result then follows from \cite[Theorem~1]{higman1961geometric}, and the fact that in these cases $|\GaL_1(q^3)|$ is equal to the number of flags of $\Gamma$.
\end{proof}

\subsection{Flag-transitive \texorpdfstring{$q$}{q}-graphs of dimension at most four} 

If a $q$-graph is $k$-regular, then the (closed) neighbourhood of a vertex forms a $(k+1)$-dimensional subspace, and thus there are no non-trivial regular $q$-graphs when $\dim(V) \leq 2$, and a graph in which neighbourhoods are of dimension $k+1 = 2$ is necessarily a disjoint union of edges. The next lemma carries this analysis a little further, and allows us to avoid the classification of transitive linear groups in dimensions $2,3,4$. 

\begin{lemma}\label{lem:smalldim}
 Let $V=\F_q^n$ and let $\Gamma=(\V,\E)$ be a non-trivial $k$-regular and vertex-transitive $q$-graph. Then at least one of the following holds.
 \begin{enumerate}[$(1)$]
  \item $n=4$ and $\E$ is the set of lines in a spread of $\pg_3(q)$, i.e. $\E$ is a partition of $\V$ into disjoint subspaces.
  \item $n=4$ and $\E$ is the set of lines of the symplectic generalised quadrangle $\W(q)$.
  \item $n\geq 5$.
 \end{enumerate}
\end{lemma}

\begin{proof}
Since the neighbourhood of a vertex is a $(k+1)$-space, the $q$-graphs with $k = 0$ and $k= n-1$ are trivial (being empty and complete, respectively). There are no non-trivial $q$-graphs with $n \leq 2$ and if $n \leq 4$ the cases to consider are $(n,k) = (3,1)$ or $(4,1)$ or $(4,2)$. 

When $n=3$, the incidence matrix of $1$- and $2$-dimensional subspaces is the incidence matrix of a projective plane, which is invertible. Apply Block's Lemma \cite{block1967orbits} to see that the number of $G$-orbits on $2$-dimensional spaces is equal to the number of orbits on $1$-dimensional spaces. By hypothesis, $\Gamma$ is $G$-vertex-transitive, hence $G$-edge-transitive and $\Gamma$ is trivial. In particular, no non-trivial $q$-graph with parameters $(n,k) = (3,1)$ is flag-transitive.
 
Suppose that $(n,k) = (4,1)$. The neighbourhood of a vertex $X$ is a $2$-space; say $N(X) = \langle x, y\rangle$, where $x,y\in V$. By regularity, every vertex $Y=\langle\alpha x + \beta y\rangle$, where $(\alpha,\beta)\in\F_q^2\setminus\{(0,0)\}$, has neighbourhood $N(X)$. Hence, $N(Y) = N(X)$ and $\Gamma$ is a disjoint union of edges, that is, $\E$ is a $2$-spread. 

Finally, suppose that $(n,k) = (4,2)$. The neighbourhood of each vertex is a $3$-space containing that vertex. Hence the map $\ast: v \mapsto N(v)$ is a symplectic polarity. Using the Fundamental Theorem of Projective Geometry, one infers that $\ast = T\circ\top$ where $\top$ is the canonical duality on $\pg_n(q)$ and $T \in \textrm{PGL}_{n}(q)$. The bilinear form given by $\langle u, v\rangle = uv^{\ast} = u(vT)^{\top} = uT^{\top} v^{\top}$ is symplectic, since $v \in v^{\ast}$ for all non-zero $v$, and the claim follows from the classification of bilinear forms.
\end{proof}

\subsection{Transitive linear groups}

In this section we recall the classification of transitive subgroups of $\GaL_n(q)$. Clearly, the automorphism group of $\Gamma$ belongs to this list. In light of Lemma \ref{lem:smalldim} we can exclude all cases with $n < 5$; we exclude also the case that $\SL_{n}(q) \leq G$ since in this case the $q$-graph is necessarily trivial. 

\begin{theorem}[cf. Theorem 3.1, \cite{giudici2023subgroups}]\label{thm:linclass}
Suppose that $H \leq \GaL_{n}(q)$ where $n \geq 5$ and suppose that $H$ acts transitively on the set of all 1-subspaces of $\mathbb{F}_{q}^{n}$. Then either $H \leq \GaL_{1}(q^n)$, or there exists $H_{0} \unlhd H$ such that $(H_0, n, q)$ are one of the following:
\begin{enumerate}[$(1)$]
\item $\left(\SL_{a}(q^{b}),\,ab,\, q\right)$ where $a > 1$; 
\item $\left(\Sp_{a}(q^{b}),\, ab,\, q\right)$ where $a$ is even; 
\item $\left(G_{2}(q^{b})',\,6b,\, q\right)$ where $q$ is even; 
\item $\left(\SL_{2}(13),\,6,\,3\right)$.
\end{enumerate}
\end{theorem}

The analysis of regular, flag-transitive $q$-graphs falls naturally into case-by-case consideration of the cases of Theorem \ref{thm:linclass}. We eliminate the final case in the following lemma, and deal with all other cases in subsequent sections. 

\begin{corollary}\label{cor:SL(2,13)}
 Let $\Gamma=(\V,\E)$ be a $k$-regular $q$-graph on $V=\F_3^6$, where $k\geq 2$, and suppose that $\SL_{2}(13) \unlhd G\leq\Aut(\Gamma)$. Then $\Gamma$ is not $G$-flag-transitive. 
\end{corollary} 

\begin{proof} 
 Note that if $\Gamma$ is $G$-flag-transitive and $X\in\V$, then the stabiliser $G_X$ must act transitively on the set of edges through $X$, and hence must have order divisible by $(3^k-1)/2$, where $2\leq k\leq 6$. Inside $\GaL_6(3)$, the subgroup $\SL_2(13)$ is its own normaliser, and hence $G=\SL_2(13)$. It follows that $G_X$ has order $6$, which is too small to act transitively on the set of edges through $X$. 
\end{proof}

\section{Examples}\label{sec:examples}

In this section, we construct various $q$-graphs, beginning with those arising in Theorem \ref{thm:main}, and then constructing some interesting examples that do not satisfy all the hypotheses of Theorem~\ref{thm:main}. 

\subsection{Flag-transitive \texorpdfstring{$q$}{q}-graphs}

\begin{example} 
The complete $q$-graph on $\mathbb{F}_{q}^n$ has every $2$-dimensional subspace as an edge; the empty $q$-graph has no edges. Each of these $q$-graphs admit an action of $G=\GaL_n(q)$, which acts transitively on ordered bases, and hence they are both $G$-symmetric and $G$-flag-transitive. 
\end{example} 

\begin{example}\label{ex:lineSpeads}
 \emph{($2$-Spreads: $k=1$.)} Let $G\leq\GaL_n(q)$ act transitively on $\V$. If $\E$ is any $G$-invariant set of $2$-spaces inducing a partition of $\V$ (that is, that form a $2$-spread of $V$), then $\Gamma=(\V,\E)$ is $1$-regular and $G$-flag-transitive. This follows from the fact that the neighbourhood of a vertex contains only one edge, and so the $G$-flag-transitive condition reduces to $G$-vertex-transitive. Note that since $\E$ forms a line-spread of $\pg_{n-1}(q)$, this implies that $n$ is even. 
 Let $\L=\{U+v\mid U\in\E,\,v\in V\}$, and note that $\L$ is the set of lines of a linear space with point-set $V$. Since the group $T_V.G$, where $T_V$ is the group of translations by elements of $V$, acts transitively on the set of all incident point-line pairs of the linear space given by $\L$, the main result of \cite{buekenhout1990linear} allows us to completely determine the possibilities for $\E$. Noting that the spreads given by the examples \cite[Section~3.2(a) and (b)]{buekenhout1990linear} are $t$-spreads, where $t$ is odd, we have that one of the following holds:
 \begin{enumerate}[$(1)$]
  \item $\E$ is a Desarguesian $2$-spread;
  \item $\E$ is the set of lines containing $0$ in a nearfield plane of order 9 (see \cite[Section~5]{foulser1964solvable}); or,
  \item $\E$ is the set of lines containing $0$ in one of two non-isomorphic linear spaces known as \emph{Hering spaces}, originally constructed in \cite{hering2020two}. 
 \end{enumerate}
 In each of the above cases, the resulting spread is an orbit of a transitive linear group $G$ in its action on $\binom{V}{2}_q$, and $G$ satisfies: $G\leq\GaL_{n/2}(q^2)$ in case (1),  $G\leq 2_-^{1+4}.{\rm A}_5\leq \GL_4(3)$ in case (2), and $G=\SL_2(13)\leq \GL_6(3)$ in case (3). 
\end{example}

Cases (2) and (3) of Example~\ref{ex:lineSpeads} are flag-transitive but not symmetric, and by the proofs of Theorem~\ref{thm:main} and Corollary~\ref{cor:main} are the only examples of non-Desarguesian spreads that give rise to flag-transitive $q$-graphs.

\begin{example}\label{ex:spreadsByProjPlanes}
\emph{(Regular spreads: $k=2$.)}
 Let $q\in\{2,8\}$, let $K=\F_{q^3}$, let $V=K^a$, so that $n = 3a$, and let $\S$ be the regular $3$-spread of $V$ given by the set of $1$-dimensional $K$-subspaces of $V$. Then each $1$-dimensional $\F_q$-subspace of $V$ (\emph{i.e.} vertex) is contained in a unique element of $\S$. Define a $q$-graph on $V$ by letting the neighbourhood of any vertex be the spread element containing it. Thus, for any vertex $X\in\V$, we have that $N(X)$ is the $K$-span $\langle X \rangle_{K}$, and so $\Gamma$ is a $2$-regular $q$-graph. 
 
 The full stabiliser of $\E$ is $\GaL_a(q^3)$, and hence $\Aut(\Gamma)=\GaL_a(q^3)$. Moreover, for any $X\in\V$ the `local' $q$-graph $\Gamma_X=(\V_X,\E_X)$, where $\V_X$ is the one-spaces of $N(X)$ and $\E_X=\{U\in\E\mid U\leq N(X)\}$, is a complete $q$-graph. Since the action of the stabiliser of $N(X)$ in $\Aut(\Gamma)$ satisfies $\Aut(\Gamma)_{N(X)}^{N(X)}\cong\GaL_1(q^3)$ and $G$ acts transitively on $\S$, Lemma~\ref{lem:completeSemiLinear} implies that $\Gamma$ is flag transitive, with parameters $(n,q,k) = (3a, 2, 2)$ or $(n,q,k) = (3a,8,2)$. Note that $\Gamma$ is also $G$-flag-transitive for $G=\GaL_1(q^{3a})$, since this also has the property that $G_{N(X)}^{N(X)}\cong\GaL_1(q^3)$. We prove in Lemma~\ref{lem:projectionIsSemilinear} that these are the only examples arising in a similar manner from $t$-spreads with $t\geq 3$.
\end{example}

\begin{example}\label{ex:symplecticPolarSpaces}
 \emph{(Symplectic polar spaces.)} Let $n\geq 2$ and let $V = \mathbb{F}_{q}^{2n}$ be equipped with a symplectic bilinear form $f$. Define a $q$-graph on $V$ by letting $\E$ be the totally isotropic $2$-spaces of $V$ and note that $\E$ is preserved by $G=\Sp_{2n}(q)$. By Theorem~\ref{thm:linclass}, $G$ acts transitively on $\V$. Following the notation in \cite[Section~3.5]{wilson2009finite}, let $X=\langle e_1\rangle$, and let $W=\langle e_1,f_1\rangle$, so that $W^\perp=\langle e_2,\ldots,e_{n},f_{n},\ldots,f_2\rangle$. Observe that $N(X)=X\oplus W^\perp$. We now have that the stabiliser $G_X$ contains a subgroup $\Sp_{2n-2}(q)$ acting transitively on $W^\perp\setminus\{0\}$, and hence on the set of edges containing $X$. Thus $\Gamma$ is $G$-flag-transitive. Note that $\Gamma$ is the $q$-analogue of a non-trivial strongly regular graph, see \cite[Example~3.13]{braun2023q}.
\end{example}

The next example involves \emph{generalised hexagons}. We give a brief description of symplectic generalised hexagons here, but for more details about generalised hexagons, see \cite[Section~10.6]{brouwer2012distance}, \cite[Section~2.4]{van2012generalized} or \cite[Section~4.3.8]{wilson2009finite}. Wilson gives the following multiplication table for a non-associative $8$-dimensional algebra over a field of characteristic $2$ (only non-zero entries are displayed).  
\[
\begin{array}{c|cccccccc}
     & x_1 & x_2 & x_3 & x_4 & x_5 & x_6 & x_7 & x_8 \\
    \hline
    x_1 &     &     &     &     & x_1 & x_2 & x_3 & x_4 \\
    x_2 &     &     & x_1 & x_2 &     &  &  x_5   & x_6 \\
    x_3 &     & x_1 &     & x_3    &  &  x_5   &  & x_7 \\
    x_4 & x_1 &     &     &  x_4   &  & x_6 & x_7 &     \\
    x_5 &     & x_2 &  x_3   &  &  x_5   &     &     & x_8 \\
    x_6 & x_2 &     & x_4    &  &  x_6   &     &   x_8  &  \\
    x_7 & x_3 & x_4 &     &     &   x_7  &  x_8   &  &  \\
    x_8 & x_5 & x_6 &  x_7   & x_8 &     &     &     &  \\
\end{array}
\]
Observe that $x_{4}+x_{5}$ is the multiplicative identity for this algebra; and define $\textrm{Tr}(\sum\lambda_{i}x_{i}) = \lambda_{4} + \lambda_{5}$.
Since we work in characteristic $2$, the subspace spanned by $U = \langle x_{1}, x_{2}, x_{3}, x_{4}+x_{5}, x_{6}, x_{7}, x_{8}\rangle$ is trace zero, and the quotient $U/\langle x_{4}+x_{5} \rangle$ admits an action of $G_{2}(2^{b})$, defined to be all elements of $\Sp_{6}(2^b)$ (restricted to this space) which preserve a trilinear form defined by 
\[ T(x,y,z) = \textrm{Tr}(xyz)\]
It can be verified that 
\[ T(x_{1}, x_{6}, x_{7}) = T(x_{2}, x_{3}, x_{8}) = 1\]
and that all other products of basis elements are zero. The generalised hexagon has as edges the three spaces which are isotropic with respect to the symplectic form, and are of trace zero; for $q > 2$ its automorphism group contains the simple group $G_{2}(q)$. 

\begin{example}\label{ex:symplecticHexagons}
 \emph{(Generalised hexagons.)} Let $q$ be even and, following \cite[Section~4.3.5]{wilson2009finite}, let $G=G_2(q)$ act on $V=\F_q^6$, where $\{x_1,x_2,x_3,x_6,x_7,x_8\}$ is a basis for $V$ (note that, since $q$ is even, we can omit $x_4,x_5$ and use the notation of \cite{wilson2009finite} otherwise unchanged). Then, letting $\V$ be the $1$-spaces of $V$ and $\E$ be the lines of the symplectic generalised hexagon (see \cite[Section~4.3.8]{wilson2009finite}) we claim that $\Gamma=(\V,\E)$ is a $2$-regular and $G$-flag-transitive $q$-graph. To see this, first note that $N(\langle x_1\rangle)=\langle x_1,x_2,x_3\rangle$. We then deduce that $\Gamma$ is $G$-flag-transitive from the fact that $G$ acts transitively on $\V$, and that the stabiliser $G_{\langle x_1\rangle}$ induces a transitive action of $\GL_2(q)$ on $\langle x_1,x_2,x_3\rangle/\langle x_1\rangle\setminus\{0\}$.
\end{example}

\subsection{Further examples of \texorpdfstring{$q$}{q}-graphs}\label{sec:furtherExamples}

In this section, we construct several examples of $q$-graphs with one of the hypotheses of Theorem~\ref{thm:main} dropped. Our first two examples can be considered to be complements of one another.

\begin{example}\label{ex:spreadComplements}
 \emph{(Symmetric, but not regular.)} Let $t$ be an integer at least $2$, let $K=\F_{q^t}$, let $V=K^a$, so that $n=at$, and let $\S$ be the regular $t$-spread of $V$ induced by its $K$-vector space structure. Define a $q$-graph $\Gamma=(\V,\E)$ on $V$ by letting $\E$ be the set of all $2$-spaces of $V$ that intersect at least two elements of $\S$ non-trivially. We then have that $\Aut(\Gamma)$ preserves $\S$, and so $\Aut(\Gamma)=\GaL_{n/t}(q^t)$. Since $\GaL_{n/t}(q^t)$ acts transitively on the set of all $K$-bases for $V$, and hence is transitive on adjacent pairs of vertices of $\Gamma$, it follows that $\Gamma$ is symmetric. However, every vertex $X$ is adjacent to $\frac{q^n-q^t}{q-1}$ other vertices, and so $t\geq 2$ implies that $N(X)$ does not form a subspace of $V$. 
\end{example}

\begin{example}\label{ex:subfields}
 \emph{(Regular and vertex-transitive, but not generally flag-transitive.)} Let $V=\F_q^n$ and let $\S$ be a regular $t$-spread of $V$, for some integer $t$ dividing $n$. Define a $q$-graph $\Gamma=(\V,\E)$ on $V$ by letting $\E$ be the set of all $2$-spaces of $V$ that intersect precisely one element of $\S$ non-trivially. We then have that $\Aut(\Gamma)$ preserves $\S$, and hence $\Aut(\Gamma)=\GaL_{n/t}(q^t)$. Thus $\Gamma$ is vertex-transitive. The fact that the neighbourhood of each vertex is the element of $\S$ containing it implies that $\Gamma$ is $(t-1)$-regular. Let $S\in\S$. Then, by Lemma~\ref{lem:completeSemiLinear}, the action $\Aut(\Gamma)_S^S\cong\GaL_1(q^t)$ does not generally act flag-transitively on the projective space induced on $S$, and hence $\Gamma$ is not flag-transitive, unless $\Gamma$ is as in Example~\ref{ex:spreadsByProjPlanes}.
\end{example}

The next example arises from a one-dimensional semi-linear group on $\F_{2^5}$ that happens to act transitively on $2$-dimensional $\F_2$-subspaces, namely $\GaL_1(2^5)$.

\begin{example}\label{ex:EdgeTransNotFlagTrans}
 \emph{(Vertex-transitive and edge-transitive, but not flag-transitive.)} 
 Let $V=\F_2^{5t}$, let $\S$ be a $5$-spread of $V$, and let $\E$ be the set of all $2$-dimensional $\F_2$-spaces of $V$ that are contained in an element $W$ of $\S$. Then $G=\Aut(\Gamma)=\GaL_t(2^5)$, and $\Gamma$ is $G$-vertex-transitive. Moreover, $G_W^W=\GaL_1(2^5)$ and so, by \cite[Theorem~3.1(b)]{giudici2023subgroups}, $G_W^W$ acts transitively on $\binom{W}{2}_2$, and hence $\Gamma$ is edge-transitive. However, by Lemma~\ref{lem:completeSemiLinear}, $\GaL_1(2^5)$ is not transitive on the set of flags contained in $W$, and so $\Gamma$ is not $G$-flag-transitive. 
\end{example}

Finally, we consider some examples obtained via field reduction. Let $F=\mathbb{F}_{q} \leq K=\mathbb{F}_{q^{t}}$ be an extension of fields, and suppose that $\Gamma=(\V,\E)$ is a $q^{t}$-graph on $V=K^n$, so that $\V=\binom{V}{1}_{q^t}$ and $\E\subseteq\binom{V}{2}_{q^t}$. Then we may form a $q$-graph $\Gamma_{K/F}=(\V_{K/F},\E_{K/F})$, where $\V_{K/F}=\binom{V}{1}_q$ and $\E_{K/F}$ is the set of all elements of $\binom{V}{2}_q$ that are contained in an edge of $\Gamma$.

\begin{example}\label{ex:fieldReduction}
 Denote by $\Gamma(n,q)$ be the $q$-graph on $\F_q^{2n}$ that is invariant under $\Sp_{2n}(q)$ given in Example ~\ref{ex:symplecticPolarSpaces}. Let $F=\F_2$, $K=\F_4$, and $L=\F_8$. Then each of the $2$-graphs $\Gamma(12,2)$, $\Gamma(6,4)_{K/F}$, and $\Gamma(4,8)_{L/F}$ are defined on $V=\F_2^{12}$ and are $k$-regular with $k=10$, $k=8$ and $k=6$, respectively. These graphs are also vertex-transitive (see Theorem~\ref{thm:linclass}), but only $\Gamma(12,2)$ is flag-transitive.
\end{example}

\section{One dimensional semi-linear case} \label{sec:semi}

In this section $V = \F_{q^n}$, considered an $n$-dimensional vector space over $\F_{q}$, and $\Gamma$ is a $q$-graph such that $\Aut(\Gamma)\leq\GaL_1(q^n)$. Since the multiplicative group $\F_q^\times$ stabilises every $\F_q$-subspace of $V$, necessarily $\Aut(\Gamma)$ contains $\F_q^\times$. Below we outline the well-known theory of transitive subgroups of $\GaL_1(q^n)$ in terms of a canonical description via certain generators. 

Let $q=p^t$, with $p$ prime, let $\omega$ be a primitive element of the $\F_{q^n}$, and let $\alpha$ be a generator of the Galois group ${\rm Gal}(\F_{q^n}/\F_p)$. Then $G\leq\GaL_1(q^n)$ can be written uniquely, see Foulser \cite[Lemma~4.1]{foulser1964flag} or \cite[Lemma~4.4]{li2009homogeneous}, as $\langle \omega^d,\alpha^s\omega^e\rangle$, where the integers $d,e,s$ satisfy the following:
\begin{enumerate}[(1)]
 \item $d>0$ and $d\mid q^n-1$.
 \item $s>0$ and $s\mid nt$.
 \item $0\leq e<d$ and $d\mid e(q^n-1)/(p^s-1)$.
\end{enumerate}
Foulser's group elements act from the left, while here we are using exponential notation, so we have replaced $\omega^e\alpha^s$ by $\alpha^s\omega^e$ in order to preserve the validity of those results. 

\begin{lemma}\label{lem:1dimDoneViaConj}
 Let $\Gamma=(\V,\E)$ be a $k$-regular and $G$-flag-transitive $q$-graph on $V=\F_{q^n}$, where $G\leq \GaL_1(q^n)$ and $k\geq 2$. Then $k=2$, $q\in\{2,8\}$, and $\Gamma$ is as in Example~\ref{ex:spreadsByProjPlanes}.
\end{lemma}

\begin{proof}
 Let $q=p^t$ where $p$ is prime, and let $\alpha$ generate the Galois group of $\mathbb{F}_{q^{n}}$ over $\mathbb{F}_{p}$.
 
 Consider the vertex $X=\F_q\leq V$, and let $\E_X$ be the set of edges through $X$. Then $\E_X$ is a set of $2$-dimensional $\F_q$-subspaces of $N(X)$, which intersect pairwise in $X$ and partition $N(X)\setminus X$, as a result $|\E_X| = q^{k-1} + \cdots + q + 1$. 

 We will apply the discussion preceding this lemma to the stabiliser $G_X$. Since $G_X$ fixes $X$ and we may assume that the scalars $\F_q^\times$ form a subgroup of $G_X$, we have that $G_X\cap\GL_1(q^n)=\F_q^\times$. Thus
 \[
  G_X=\langle\omega^{(q^n-1)/(q-1)}, \alpha^s\omega^e\rangle,
 \]
 for some integers $e,s$ with $0\leq e<(q^n-1)/(q-1)$. Since $1^g\in X\setminus\{0\}$ for any $g\in G_X$, we deduce that $1^{\alpha^s\omega^e}=\omega^e\in\F_q^\times$. Hence $e$ is a multiple of $(q^n-1)/(q-1)$, and the condition $0\leq e<(q^n-1)/(q-1)$ then implies that $e=0$. It follows that 
 \[
  G_X=\F_q^\times\rtimes\langle \alpha^s\rangle.
 \]
 Since elements of $\E_X$ are $\F_{q}$-subspaces, they are invariant under $\F_q^\times\lhd G_X$. By the definition of flag-transitivity, $H=\langle \alpha^s\rangle$ acts transitively on $\E_X$. By the orbit stabiliser-theorem,  $|\E_X|=(q^k-1)/(q-1)$ divides $|H|$. 
 
 Let $v\in N(X)\setminus X$ be a vector. Let $q_0=p^{s}$, so that $v^H=\{v,v^{q_0},\ldots,v^{q_0^{\ell-1}}\}$ for some $\ell$ that is divisible by $(q^k-1)/(q-1)$. By \cite[Theorem~3.33]{lidl1997finite}, the minimal polynomial for $v$ over $\F_{q_0}$ has degree $\ell$, and $\langle v^H\rangle_{\F_{q_0}}=\F_{q_0^\ell}$. Let $r=\gcd(t,s)$, so that $\F_{p^r}$ is the largest field contained in both $\F_q$ and $\F_{q_0}$. Then, since $\langle v^H\rangle_{\F_{p^r}}\leq\langle v^H\rangle_{\F_q}\leq N(X)$, we deduce that $N(X)$ has $\F_q$-dimension at least $\ell r/t$, or at least $\ell/t$, so that $k+1\geq \ell/t$. Hence $(k+1)t\geq q^{k-1}+\cdots+q+1$. The equation of the previous sentence holds only when $k=p=2$ and $t\in\{1,2,3\}$. 
 
 The cases $t=1$ and $t=3$ arise in Example~\ref{ex:spreadsByProjPlanes}. Moreover, the parameters and the argument in this proof directly imply that for $t=1$ we have $N(X)=\F_8$, and for $t=3$ we have $N(X)=\F_{2^9}$. Hence $\Gamma$ is as in Example~\ref{ex:spreadsByProjPlanes} when $t\in\{1,3\}$. Suppose that $k=p=t=2$. Then $\F_{32}$ is contained in $N(X)$, and $N(X)$ should be an $\F_4$-vector space of dimension $3$. However, $\langle \F_{32}\rangle_{\F_4}$ has $\F_4$-dimension $5$, giving a contradiction and completing the proof.
\end{proof}

\section{Groups of Lie type} \label{sec:insol}

Throughout this section, $n = ab$ is the dimension of $V$ over $F=\F_q$ and we consider an embedding of a classical group $G$ in dimension $a$ defined over $K=\F_{q^b}$ into $\GaL_{ab}(q)$. In light of Theorem~\ref{thm:linclass}, the cases to consider are $\SL_{a}(q^{b}) \unlhd G$ and $\Sp_{a}(q^{b}) \unlhd G$ for all prime powers $q$; and $G_{2}(q^{b})'\unlhd G$ where $q$ is even. Note that the case with $b = 1$ (i.e. the embedding is natural) yields trivial $q$-graphs in the first case, and strongly regular $q$-graphs in the second case. 

\subsection{Extension fields}

The purpose of the following lemma will be in dealing with the case that a neighbourhood of a vertex is contained in its $K$-span, for some extension field $K$ of $F=\F_q$, that arises for each family of groups in this section. Lemma~\ref{lem:1dimDoneViaConj} will enable us to completely characterise such $q$-graphs.

\begin{lemma}\label{lem:projectionIsSemilinear}
 Let $F=\F_q$, let $K=\F_{q^b}$ with $2\leq k\leq b$, let $V=K^a$ so that $n=ab$, and let $\Gamma=(\V,\E)$ be a $k$-regular and $G$-flag-transitive $q$-graph on $V$, where $G\leq\GaL_a(q^b)$. Furthermore, suppose that there exists some $X\in \V$ such that $N(X)$ is contained in the $K$-span $\langle X\rangle_K$. Then $\Gamma$ is as in Example~\ref{ex:spreadsByProjPlanes}. 
\end{lemma}

\begin{proof}
 Let $W=\langle X\rangle_K$ and let $\Gamma_W$ be the $q$-graph on $W$ with edge set $\E_W=\{U\in \E\mid U\leq W\}$. Since $G$ is a subgroup of $\GaL_a(q^b)$, we have that the partition of $V\setminus\{0\}$ induced by the $1$-dimensional $K$-subspaces of $V$ is preserved by $G$. Together with the fact that $G$ acts transitively on $\V$, this implies that $\binom{W}{1}_q$ is a block of imprimitivity for the action of $G$ on $\V$. In particular, $G_W$ acts transitively on $\binom{W}{1}_q$. Moreover, since $G_X$ fixes the $K$-span of $X$, that is, fixes $W$, we deduce that $G_X\leq G_W$. By the flag-transitivity of $\Gamma$, we have that $G_X^W$ acts transitively on the set of edges of $\Gamma_W$ incident with $X$, and thus $\Gamma_W$ is $G_W^W$-flag-transitive. Since the faithful action of the stabiliser of $W$ inside $\GaL_a(q^b)$ is isomorphic to $\GaL_1(q^b)$, we have that $G_W^W\leq \GaL_1(q^b)$ and hence, by Lemma~\ref{lem:1dimDoneViaConj}, $\Gamma_W$ is as in Example~\ref{ex:spreadsByProjPlanes}. The fact that $X$ was chosen arbitrarily then implies that the set of neighbourhoods of $\Gamma$ are precisely the elements of the $3$-spread of $V$ preserved by $\GaL_a(q^b)$. It follows that $\Gamma$ itself is as in Example~\ref{ex:spreadsByProjPlanes}, the result holds.
\end{proof}

\subsection{Special linear groups over extension fields}

\begin{lemma}\label{lem:SL}
 Let $F=\F_q$, let $K=\F_{q^b}$ where $b>1$, and suppose that $\Gamma$ is a non-trivial, $k$-regular and $G$-flag-transitive $q$-graph such that $\SL_{a}(q^{b})\unlhd G$, where $n=ab$. Then $\Gamma$ is as in Example~\ref{ex:spreadsByProjPlanes}.
\end{lemma}

\begin{proof}
 Let $\Gamma=(\V,\E)$, let $e_1,\ldots,e_a$ be a $K$-basis for $V$, let $X=\langle e_1\rangle_F\in \V$, and let $U=\langle X\rangle_K$. The stabiliser $G_X$ contains a subgroup of shape 
\[ q^{b(a-1)}.((q-1)\times\SL_{a-1}(q^b))\,,\]
 with elements of the form
 \[
  A=
  \begin{bmatrix}
   \lambda & \b 0\\
   \b u^\top & M
  \end{bmatrix}
 \]
 where $\lambda\in F^\times$, $\b u\in K^{a-1}$ and $\det(M) = \lambda^{-1}$ (note that matrices are acting on the right here via $x\mapsto xA$). First note that $U$ is an $FG_X$-submodule of $V$ and, since $b>1$, $X$ is a proper $FG_X$-submodule of $U$. We claim that any $FG_X$-submodule of $V$ intersecting $V\setminus U$ non-trivially also contains $U$, in which case Lemma~\ref{lem:transOnQuotientSpace} implies that $N(X)$ is a submodule of $U$. Let $v\in V\setminus U$. Since $G_X$ contains a copy of $\SL_{a-1}(q^b)$ that fixes $e_1$ and acts transitively on the non-zero vectors of $\langle e_2,\ldots,e_a\rangle$, we may assume that $v=\alpha e_1+e_2$ for some $\alpha\in K$. Let $\b u=(\beta,0\ldots,0)$, where $\beta\in K$, and let
 \[
  B=
  \begin{bmatrix}
   1 & \b 0\\
   \b u^\top & I_{a-1}
  \end{bmatrix}.
 \]
 Then $B\in G_X$ and $vB=(\alpha+\beta) e_1+e_2$. Thus, any $FG_X$-submodule $U'\leq V$ with $v\in U'$ also contains $v-vB=-\beta e_1$. Since $\beta\in K$ was arbitrary, we have that $U\leq U'$, and the result then follows from Lemma~\ref{lem:projectionIsSemilinear}.
\end{proof}

\subsection{Symplectic groups over extension fields}

Let $F=\F_q$, let $K=\F_{q^b}$, and let $f$ be a symplectic form on $V=K^a$, where $a$ is even. If $b=1$, then taking the edge-set $\E$ to the set of all isotropic $2$-spaces gives a strongly regular $q$-graph, as in Example~\ref{ex:symplecticPolarSpaces}. In this section, we show that these are the only examples of $G$-flag-transitive, regular $q$-graphs with $\Sp_a(q^b)\unlhd G$. 

\begin{lemma}\label{lem:Sp}
 Let $F=\F_q$, let $K=\F_{q^b}$, and suppose that $\Gamma$ is a non-trivial, $k$-regular and $G$-flag-transitive $q$-graph of order $n$, where $\Sp_{a}(q^{b})\unlhd G$, $a$ is even and at least $4$, $n=ab$, and $k\geq 2$. Then either $b=1$ and the neighbourhood of $X$ is equal to its dual under a symplectic polarity (as in Example~\ref{ex:symplecticPolarSpaces}), or $b>1$ and $\Gamma$ is as in Example~\ref{ex:spreadsByProjPlanes}.
\end{lemma}

\begin{proof}
 Following the notation in \cite[Section~3.5]{wilson2009finite}, let $e_1,\ldots,e_{a/2},f_{a/2},\ldots,f_1$ be a symplectic $K$-basis for $V$. Assume $X=\langle e_1\rangle_F$, and let $W=\langle e_1,f_1\rangle_K$, so that 
 \[
  W^\perp=\langle e_2,\ldots,e_{a/2},f_{a/2},\ldots,f_2\rangle_K
 \]
 and $V=W\oplus W^\perp$. 

 First, suppose that $b=1$, so that $F=K$. Then there is a copy of $\Sp_{a-2}(q)$ inside $G_X$ that acts transitively on $W^\perp\setminus\{0\}$, and hence $W^\perp$ is an irreducible $FG_X$-submodule of $V$. Setting $N(X)=X\oplus W^\perp$, which is equal to $X^\perp$, we obtain Example~\ref{ex:symplecticPolarSpaces}. Since $V=W\oplus W^\perp$, we see that $X\oplus W^\perp$ is the only $FG_X$-submodule of $V$ that contains $X$ and has dimension at least $3$, and so Example~\ref{ex:symplecticPolarSpaces} is the only example arising for $b=1$.

 Now assume that $b>1$. Let $Y=\langle X\rangle_K$, and note that $Y$ is an $FG_X$-submodule of $V$. We claim that any $FG_X$-submodule $U$ of $V$ that intersects $V\setminus Y$ non-trivially contains $Y$. If this holds, then it follows from Lemma~\ref{lem:transOnQuotientSpace} that $N(X)\leq Y$. To prove the claim, assume $U$ is such a submodule of $V$. Recall that $V=W\oplus W^\perp$ and that $G_X$ contains a copy of $\Sp_{a-2}(q^b)$ that fixes $e_1$ and $f_1$ and acts transitively on $W^\perp\setminus\{0\}$. Thus we may assume that $u=\alpha e_1+\beta f_1+e_2\in U$, where $\alpha,\beta\in K$. For each $\lambda\in K$, $G_X$ contains the elements $y_{12}(\lambda)$, as in \cite[Eq.~(3.24)]{wilson2009finite}, mapping $f_1\mapsto f_1+\lambda e_2$, $f_2\mapsto f_2+\lambda e_1$, and acting trivially on the remaining basis elements. Thus, $u^{y_{12}(\beta^{-1})}-u=e_2$ is contained in $U$. Also, for each $\lambda\in K$, we have that $G_X$ contains the elements $x_{12}(\lambda)$, as in \cite[Eq.~(3.24)]{wilson2009finite}, mapping $f_1\mapsto f_1+\lambda f_2$, $e_2\mapsto e_2+\lambda e_1$, and acting trivially on the remaining basis elements. It follows that $e_2^{x_{12}(\lambda)}-e_2=\lambda e_1$ is in $U$, for each $\lambda\in K$, and thus $Y\leq U$. Applying Lemma~\ref{lem:projectionIsSemilinear} completes the proof.
\end{proof}

\subsection{Lie-type group \texorpdfstring{$G_2(q)$}{G2(q)}}

While the group $G_{2}(q)$ acts naturally on a 7-dimensional space, in characteristic $2$ it has an exceptional embedding into 6 dimensions, as a subgroup of $\Sp_{6}(q)$. As a result, the edges of a $G_{2}(q)$-invariant $q$-graph will be contained in those of the corresponding symplectic $q$-graph, described in the previous section.

\begin{lemma}\label{lem:G2}
 Let $F=\F_q$ with $q$ even, let $K=\F_{q^b}$, and suppose that $\Gamma$ is a non-trivial, $k$-regular and $G$-flag-transitive $q$-graph of order $n$, where $G_2(q^{b})'\unlhd G$, $n=6b$, and $k\geq 2$. Then either $b=1$ and $\Gamma$ is as in Example~\ref{ex:symplecticHexagons}, or $b>1$ and $\Gamma$ is as in Example~\ref{ex:spreadsByProjPlanes}.
\end{lemma}

\begin{proof}
 Let $F=\F_q$ and let $K=\F_{q^b}$. We will follow the description of $G_2(q^b)$ given by Wilson in \cite[Section~4.3.5]{wilson2009finite}. Noting that when $q$ is even $G_2(q^b)$ embeds into $\Sp_6(q^b)$, and this allows us to work with $V=\langle x_1,x_2,x_3,x_6,x_7,x_8\rangle_K$ (effectively ignoring $x_4,x_5$). Also, we will be particularly interested in the generators $A(\lambda),B(\lambda),C(\lambda),D(\lambda),E(\lambda),F(\lambda)$ of the unipotent subgroup $U$ of $G$ as given in \cite[(4.34)]{wilson2009finite}). It is worth pointing out here that the unipotent subgroup $U$ is contained in $G_X$, and also that $1=-1$, since we are working in characteristic $2$. 
 
Let $X=\langle x_1\rangle_F$, let  $Y=\langle x_1,x_2,x_3\rangle_K$, let $Z=\langle x_1\rangle_K$, and note that under the symplectic form on $V$ that is preserved by $G$ we have $Z^\perp=\langle x_1,x_2,x_3,x_6,x_7\rangle_K$. Since $G_X$ preserves $Y$, we have that the following flag is preserved by $G_X$:
 \[
  Z< Y< Z^\perp< V.
 \]
 Suppose $b=1$, so that $F=K$ and $X=Z$. Then $q^{2+3}:\GL_2(q)\leq G_X$ (see \cite[Section~4.3.5]{wilson2009finite}). Moreover, since $\GL_2(q)$ acts transitively on the non-zero vectors of each of the quotient spaces $Y/X$ and $X^\perp/Y$, we deduce that $V$ is a uniserial $FG_X$-module with $\b 0<X< Y<  X^\perp< V$ being the corresponding maximal chain of submodules. By Lemma~\ref{lem:transOnQuotientSpace}, we have that $N(X)=Y$ and hence $\Gamma$ is as in Example~\ref{ex:symplecticHexagons}. 
 
 We assume from now on that $b\geq 2$, in which case $G_X$ no longer contains $\GL_2(q^b)$. Let $W$ be an $FG_X$-submodule of $V$ such that $W\cap (V\setminus Z)$ is non-trivial. Then we claim that $Z\leq W$. Let $v_1=u_1+\beta_8 x_8$, where $u_1\in Z^\perp$ and $\beta_8\in K^\times$, so that $v_1$ can be considered to be an arbitrary element of $V\setminus Z^\perp$. Then $v_1^{A(1)}+v_1=\beta_1 x_1+\beta_8 x_2\in W$ for some $\beta_1\in K$. Hence $v_1\in W$ implies that $v_1$ intersects $Y\setminus Z$ non-trivially. Let $v_2=u_2+\beta_6x_6+\beta_7x_7\in Z^\perp\setminus Y$, where $\beta_6,\beta_7\in K$ and $(\beta_6,\beta_7)\neq(0,0)$. Then $v_2^{C(1)}+v_2=\beta_6x_2+\beta_7x_3$. Thus, $v_2\in W$ implies $W$ intersects $Y\setminus Z$ non-trivially, and hence there exists some $v_3\in (W\cap Y)\setminus Z$. Let $v_3=\gamma_1x_1+\gamma_2x_2+\gamma_3x_3$, where $\gamma_i\in K$ for $i\in\{1,2,3\}$ and $(\gamma_2,\gamma_3)\neq(0,0)$. Suppose that $\gamma_3\neq 0$. Then $v_3^{E(1)}+v_3=\gamma_3 x_2$, which implies that, upon possibly replacing $v_3$ with $\gamma_3 x_2$, we may assume that $\gamma_2\neq 0$ and $\gamma_3=0$. Supposing we have done such a substitution (if necessary), then for each $\lambda\in K$ we have $v_3^{F(\lambda)}+v_3=\lambda\gamma_2 x_1\in W$, and hence $W$ contains $Z$. The result then follows from Lemma~\ref{lem:projectionIsSemilinear}.
\end{proof}

\section{Main proofs}\label{sec:mainproof}

We are now in a position to prove our main results, first dealing with regular, flag-transitive $q$-graphs.

\begin{proof}[Proof of Theorem~\ref{thm:main}]
 By assumption, $\Gamma=(\V,\E)$ is a $k$-regular and $G$-flag-transitive $q$-graph on $V=\F_q^n$ where $\E$ is neither empty, nor the set of all $2$-spaces of $V$. If $k=1$, then the neighbourhood of each vertex is a $2$-space, so that each vertex is contained in precisely one edge, and $\Gamma$ is as in Example~\ref{ex:lineSpeads}, and part (1) holds. Assume from now on that $k\geq 2$. If $n\leq 4$, then Lemma~\ref{lem:smalldim} implies that part (3) holds, and so we may now also assume that $n\geq 5$. Since $\Gamma$ is $G$-flag-transitive, it follows that $G$ acts transitively on the $1$-spaces of $V$, and hence Theorem~\ref{thm:linclass} may be applied. If $G\leq \GaL_1(q^n)$ then it follows from Lemma~\ref{lem:1dimDoneViaConj} that $k=2$, $q\in\{2,8\}$ and $\Gamma$ is as in part (2). In the remainder of the proof we consider each of the cases Theorem~\ref{thm:linclass}(1)--(4). 
 
 Suppose that $n=ab$, for some integer $a>1$, and that $\SL_a(q^b)\unlhd G$. Note that if $b=1$ then $G$ acts transitively on ${\binom{V}{2}}_q$, and so we may also assume that $b>1$. By Lemma~\ref{lem:SL}, $\Gamma$ is as in part (2). Suppose that $n=ab$, with $a$ even, $\Sp_a(q^b)\trianglelefteqslant G$, and let $K=\F_{q^b}$. By Lemma~\ref{lem:Sp}, part (2) or part (3) holds. Next, suppose that $n=6b$, $q$ is even, $G_2(q^b)'\unlhd G$, and $K=\F_{q^b}$. Then Lemma~\ref{lem:G2} implies that part (2) or part (4) holds. Finally, the case $q=3$, $n=6$ and $\SL_2(13)\unlhd G$ is ruled out by Corollary~\ref{cor:SL(2,13)}, completing the proof.
\end{proof}

Next we prove our main result concerning regular, symmetric $q$-graphs. 

\begin{proof}[Proof of Corollary~\ref{cor:main}]
 Suppose that $\Gamma$ is a $k$-regular and $G$-symmetric $q$-graph on $V=\F_q^n$. Then $\Gamma$ is necessarily $G$-flag-transitive, so we may apply Theorem~\ref{thm:main}. Note that since a pair of vertices define a unique $2$-space in $V$ and $G$ is contained in $\GaL_n(q)$, the stabiliser $G_U$ of an edge $U\in\E$ must act $2$-transitively on the set of vertices incident with $U$. Conversely, given that $\Gamma$ is $G$-flag-transitivity, this condition is sufficient to imply that $\Gamma$ is $G$-symmetric. The edge stabiliser $G_U$ is $2$-transitive on the vertices incident with $U$ if and only if either $\Gamma$ is as in Theorem~\ref{thm:main}(1) with $q=2$ whence $G_U^U=\GaL_1(4)\cong\s_3$, or as in Theorem~\ref{thm:main}(3) or (4), where $\SL_2(q)\leq G_U^U$.
\end{proof}

Finally, we complete the paper by proving that every regular, flag-transitive $q$-graph has a symmetric graph as its classical counterpart.

\begin{proof}[Proof of Proposition~\ref{prop:CCsymmetric}]
 By Theorem~\ref{thm:main} and Corollary~\ref{cor:main}, either $\Gamma$ is as in Example~\ref{ex:lineSpeads} or~\ref{ex:spreadsByProjPlanes}, in which case $\Gamma_C$ is a union of cliques and thus a symmetric graph, or $\Gamma$ is symmetric, and $\Gamma_C$ is symmetric by Corollary~\ref{cor:classicalSymmetric}.
\end{proof}

\vspace{0.5cm}
\noindent
\textbf{Acknowledgments:} Daniel Hawtin has been supported in part by the Croatian Science Foundation under the project HRZZ-IP-2022-10-4571, by the European Union under the NextGenerationEU program, and by funding from DCU and University College Dublin. Padraig \'{O} Cath\'{a}in was supported by funding from the DCU Faculty of Humanities and Social Sciences.

Part of this work was carried out during the first Dublin Discrete Mathematics Workshop held at Dublin City University (DCU) in June 2025, during which the authors benefitted from helpful discussions with Dr Ronan Egan and Dr Andrea {\v S}vob. The authors gratefully acknowledge these discussions, and helpful feedback from Prof. Dean Crnkovi\'c on a draft of this paper.

\bibliographystyle{plain}  
\bibliography{ref}

@book{van2012generalized,
  title={Generalized polygons},
  author={Van Maldeghem, Hendrik},
  year={2012},
  publisher={Springer Science \& Business Media}
}

@book{brouwer2012distance,
  title={Distance-Regular Graphs},
  author={Brouwer, Andries E.~ and Cohen, Arjeh M.~ and Neumaier, Arnold},
  volume={18},
  year={2012},
  publisher={Springer Science \& Business Media}
}

@article{giudici2023subgroups,
  title={Subgroups of classical groups that are transitive on subspaces},
  author={Giudici, Michael and Glasby, Stephen P.~ and Praeger, Cheryl E.~},
  journal={Journal of Algebra},
  volume={636},
  pages={804--868},
  year={2023},
  publisher={Elsevier}
}

@book{dixon1996permutation,
  title={Permutation groups},
  author={Dixon, John D.~ and Mortimer, Brian},
  volume={163},
  year={1996},
  publisher={New York: Springer}
}

@article{devillers2000d,
  title={$d$-{H}omogeneous and $d$-ultrahomogeneous linear spaces},
  author={Devillers, Alice},
  journal={Journal of Combinatorial Designs},
  volume={8},
  number={5},
  pages={321--329},
  year={2000},
  publisher={Wiley Online Library}
}

@article{bamberg2021partial,
  title={Partial linear spaces with a rank 3 affine primitive group of automorphisms},
  author={Bamberg, John and Devillers, Alice and Fawcett, Joanna B.~ and Praeger, Cheryl E.~},
  journal={Journal of the London Mathematical Society},
  volume={104},
  number={3},
  pages={1011--1084},
  year={2021},
  publisher={Wiley Online Library}
}

@article{higman1962flag,
  title={Flag-transitive collineation groups of finite projective spaces},
  author={Higman, Donald G.~},
  journal={Illinois Journal of Mathematics},
  volume={6},
  number={3},
  pages={434--446},
  year={1962},
  publisher={Duke University Press}
}

@book{lidl1997finite,
  title={Finite Fields},
  series={Encyclopedia of Mathematics and its Applications},
  author={Lidl, Rudolf and Niederreiter, Harold},
  volume={33},
  year={1997},
  publisher={Elsevier Science Publishing Company, Inc.}
}

@article{praeger1993nan,
  title={An {O}'{N}an-{S}cott theorem for finite quasiprimitive permutation groups and an application to 2-arc transitive graphs},
  author={Praeger, Cheryl E.~},
  journal={Journal of the London Mathematical Society},
  volume={2},
  number={2},
  pages={227--239},
  year={1993},
  publisher={Oxford University Press}
}

@article{higman1961geometric,
  title={Geometric ${A}{B}{A}$-groups},
  author={Higman, Donald G.~ and McLaughlin, Jack E.~},
  journal={Illinois Journal of Mathematics},
  volume={5},
  number={3},
  pages={382--397},
  year={1961},
  publisher={Duke University Press}
}

@article{hawtin2022non,
  title={The non-existence of block-transitive subspace designs},
  author={Hawtin, Daniel R.~ and Lansdown, Jesse},
  journal={Combinatorial Theory, 2 (1)},
  year={2022}
}

@article{conder2009refined,
  title={A refined classification of symmetric cubic graphs},
  author={Conder, Marston and Nedela, Roman},
  journal={Journal of Algebra},
  volume={322},
  number={3},
  pages={722--740},
  year={2009},
  publisher={Elsevier}
}

@article{potovcnik2014order,
  title={On the order of arc-stabilisers in arc-transitive graphs with prescribed local group},
  author={Poto{\v{c}}nik, Primo{\v{z}} and Spiga, Pablo and Verret, Gabriel},
  journal={Transactions of the American Mathematical Society},
  volume={366},
  number={7},
  pages={3729--3745},
  year={2014}
}

@article{ivanov1993finite,
  title={On finite affine 2-arc transitive graphs},
  author={Ivanov, Alexander A.~ and Praeger, Cheryl E.~},
  journal={European Journal of Combinatorics},
  volume={14},
  number={5},
  pages={421--444},
  year={1993},
  publisher={Elsevier}
}

@article{li2009homogeneous,
  title={Homogeneous factorisations of complete graphs with edge-transitive factors},
  author={Li, Cai Heng and Lim, Tian Khoon and Praeger, Cheryl E.~},
  journal={Journal of Algebraic Combinatorics},
  volume={29},
  number={1},
  pages={107--132},
  year={2009},
  publisher={Springer}
}

@article{foulser1964flag,
  title={The flag-transitive collineation groups of the finite Desarguesian affine planes},
  author={Foulser, David A.~},
  journal={Canadian Journal of Mathematics},
  volume={16},
  pages={443--472},
  year={1964},
  publisher={Cambridge University Press}
}

@article{block1967orbits,
  title={On the orbits of collineation groups},
  author={Block, Richard E.~},
  journal={Mathematische Zeitschrift},
  volume={96},
  number={1},
  pages={33--49},
  year={1967},
  publisher={Springer-Verlag Berlin/Heidelberg}
}

@article{braun2023q,
  title={$q$-{A}nalogs of strongly regular graphs},
  author={Braun, Michael and Crnkovi{\'c}, Dean and De Boeck, Maarten and Crnkovi{\'c}, Vedrana Mikuli{\'c} and {\v{S}}vob, Andrea},
  journal={Linear Algebra and its Applications},
  number={693},
  pages={362--373},
  year={2024}
}

@book{wilson2009finite,
  title={The finite simple groups},
  author={Wilson, Robert},
  volume={251},
  series={Graduate Texts in Mathematics},
  year={2009},
  publisher={Springer}
}

@article{crnkovic2025qdeza,
  title={$q$-{A}nalogs of divisible design graphs and {D}eza graphs},
  author={Crnkovi{\'c}, Dean and De Boeck, Maarten and Pavese, Francesco and {\v{S}}vob, Andrea},
  journal={Journal of Combinatorial Theory Series A },
  volume={216},
  number={106047},
  pages={1--18},
  year={2025}
}

@article{crnkovic2025trans,
  title={Transitive regular $q$-analogs of graphs},
  author={Crnkovi{\'c}, Dean and Crnkovi{\'c}, Vedrana Mikuli{\'c} and {\v{S}}vob, Andrea and {\v{Z}}utolija, Matea Zubovi{\'c}},
  journal={Art of Discrete and Applied Mathematics},
  volume={8},
  number={1.05},
  pages={1--7},
  year={2025}
}

@article{foulser1964solvable,
  title={Solvable flag transitive affine groups},
  author={Foulser, David A.~},
  journal={Mathematische Zeitschrift},
  volume={86},
  number={3},
  pages={191--204},
  year={1964},
  publisher={Springer}
}

@incollection{hering2020two,
  title={Two new sporadic doubly transitive linear spaces},
  author={Hering, Christoph},
  booktitle={Finite geometries},
  pages={127--129},
  year={1985},
  publisher={CRC Press}
}

@article{buekenhout1990linear,
  title={Linear spaces with flag-transitive automorphism groups},
  author={Buekenhout, Francis and Delandtsheer, Anne and Doyen, Jean and Kleidman, Peter B.~ and Liebeck, Martin W.~ and Saxl, Jan},
  journal={Geometriae Dedicata},
  volume={36},
  number={1},
  pages={89--94},
  year={1990}
}

@article{buratti,
  title={Graph decompositions in projective geometries},
  author={Buratti, Marco and Naki{\' c}, Anamari and Wassermann, Alfred},
  journal={Journal of Combinatorial Designs},
  volume={29},
  number={3},
  pages={141--174},
  year={2021}
}

@article{buratti2,
  title={Designs over finite fields by difference methods},
  author={Buratti, Marco and Naki{\' c}, Anamari},
  journal={Finite Fields and Their Applications},
  volume={57},
  number={},
  pages={128--138},
  year={2019}
}

@article{segre,
  title={Teoria di Galois, fibrazioni proiettive e geometrie non desarguesiane},
  author={Segre, Beniamino},
  journal={Annali di Matematica Pura ed Applicata},
  volume={64},
  number={4},
  pages={1--76},
  year={1964}
}
\end{document}